\documentclass[11pt]{amsart}

\usepackage[T1]{fontenc}
\usepackage{mathrsfs,amsmath, amscd, amsthm,amssymb, amsfonts, verbatim,subfigure, enumerate}
\usepackage[mathscr]{euscript}
\usepackage[all,cmtip]{xy}
\usepackage{graphicx}
\usepackage[top=1in, bottom=1.25in, left=1.3in, right=1.3in]{geometry}
\usepackage{tikz-cd}

\usepackage{mathpazo}
\usepackage[colorlinks=true,linkcolor=blue,citecolor=red]{hyperref}
\usepackage{tikz}
\usepackage{amsrefs}

\newcommand{\fun}{\mathbb{F}_1}
\newcommand{\set}[2]{\left\{#1~\middle|~#2\right\}}
\renewcommand{\H}{{\mathbb{H}}}
\newcommand{\nc}{\newcommand}
\nc{\mc}{\mathcal}
\nc{\on}{\operatorname}
\nc{\ol}{\overline}
\nc{\Eh}{\on{CH}}
\nc{\wt}{\widetilde}

\nc{\F}{\mc{F}}
\nc{\M}{\on{M}}
\nc{\T}{\mc{T}}
\renewcommand{\H}{\on{H}}
\nc{\G}{\mc{G}}
\nc{\ov}{\overline}
\nc{\VFun}{\on{Vect}(\fun)}
\nc{\FF}{\mathbb{F}}
\nc{\ses}[3]{  #1 \hookrightarrow #2 \twoheadrightarrow #3 }
\nc{\Mat}{{\mathbf{Mat}}}
\nc{\f}{\mathbf{f}}
\nc{\g}{\mathbf{g}}
\nc{\E}{\mc{E}}
\nc{\EE}{\mathfrak{E}}
\nc{\MM}{\mathfrak{M}}
\nc{\pSet}{\mathcal{S}et_*}
\nc{\MS}{\mathbf{Mat}_{\bullet}}

\nc{\GG}{\mathbb{G}}

\nc{\BB}{\mathbb{B}}
\nc{\Emb}{\BB mod_{Emb}}
\renewcommand{\L}{\mc{L}}

\nc{\Set}{\mathcal{S}et_{\bullet}}
\renewcommand{\SS}{\mathcal{S}}

\theoremstyle{remark}

\allowdisplaybreaks[4]  

\usepackage{secdot}  

\theoremstyle{definition}
\newtheorem{mydef}{\textbf{Definition}}[section]
\newtheorem{myeg}[mydef]{\textbf{Example}}

\newtheorem{mythm}[mydef]{\textbf{Theorem}}

\newtheorem{rmk}[mydef]{\textbf{Remark}}
\theoremstyle{plain}

\newtheorem*{nothma}{\textbf{Theorem A}}
\newtheorem*{nothmb}{\textbf{Theorem B}}
\newtheorem*{nothmc}{\textbf{Theorem C}}
\newtheorem*{nothmd}{\textbf{Theorem D}}

\newtheorem{lem}[mydef]{\textbf{Lemma}}
\newtheorem{pro}[mydef]{\textbf{Proposition}}

\DeclareMathOperator{\id}{\textrm{id}}

\DeclareMathOperator{\Miso}{\mathcal{M}_{iso}}

\newcommand{\com}[1]{\ignorespaces}

\renewcommand{\H}{{\mathbb{H}}}
\nc{\h}{\mathfrak{h}}
\nc{\ch}{\on{CH}}

\nc{\C}{\mathcal{C}}
\nc{\fl}{\mathbf{FL}}
\renewcommand{\H}{\on{H}}
\nc{\FFF}{\mathbb{F}}

\begin{document}

 \title{Proto-exact categories of matroids, Hall algebras, and K-theory}
 \author{Christopher Eppolito,  Jaiung Jun,  and Matt Szczesny}
  \date{}
  \thanks{}
  \date{}
  \begin{abstract}
  	\normalsize{This paper examines the category $\MS$ of pointed matroids and strong maps from the point of view of Hall algebras. We show that $\MS$ has the structure of a finitary proto-exact category - a non-additive generalization of exact category due to Dyckerhoff-Kapranov. We define the algebraic K-theory $K_* (\MS)$ of $\MS$ via the Waldhausen construction, and show that it is non-trivial, by exhibiting injections $$\pi^s_n (\mathbb{S}) \hookrightarrow K_n (\MS)$$ from the stable homotopy groups of spheres for all $n$. Finally, we show that the Hall algebra of $\MS$ is a Hopf algebra dual to Schmitt's matroid-minor Hopf algebra. }
	
  \end{abstract}
  
\maketitle

\section{Introduction}

In this paper we examine the category of pointed matroids and strong maps from the perspective of Hall algebras, which have traditionally been studied in representation theory. This perspective sheds new light on certain combinatorial Hopf algebras built from matroids, and opens the door to defining algebraic K-theory of matroids. The rest of this introduction is devoted to introducing the main actors. 

\subsection{Hall algebras of Abelian and exact categories}

The study of Hall algebras is by now a well-established area with several applications in representation theory and algebraic geometry (see \cite{S} for a very nice overview). We briefly recall the generic features of the most basic version of this construction. Given an abelian category $\mc{C}$, let 
\[
{\fl}_{i}(\mc{C}) := \{ A_0 \subseteq A_1 \subseteq \cdots \subseteq A_i \vert \; A_k \in \on{Ob}(\C) \}
\]
denote the stack parametrizing isomorphism classes of flags of objects in $\mc{C}$ of length $i+1$ (viewed here simply as a set). Thus $${\fl}_0 (\C) = \on{Iso}(\C), $$ the moduli stack of isomorphism classes of objects of $\C$, and $${\fl}_1(\C) = \{ A_0 \subseteq A_1 \vert \; A_0, A_1 \in \on{Ob}(\C) \}$$ is the usual Hecke correspondence. We have maps 
\begin{equation} \label{Hecke_correspondence}
\pi_i : \fl_{1} (\C) \rightarrow \on{Iso}(\C), \; i=1, 2, 3,
\end{equation}
where
\begin{align*}
\pi_1 (A_0 \subseteq A_1) &= A_0, \\
\pi_2 (A_0 \subseteq A_1) &= A_1, \\
\pi_3 (A_0 \subseteq A_1) &= A_1/A_0. \\
\end{align*}
We may then attempt to define the Hall algebra of $\C$ as the space of $\mathbb{Q}$-valued functions on $\on{Iso} (\C)$ with finite support, i.e.
\[
\on{H}_{\C} = \mathbb{Q}_c[\on{Iso} (\C)]
\]
with the convolution product defined for $f, g \in \on{H}_{\C}$
\begin{equation} \label{Hallprod}
f \star g := \pi_{2 *}(\pi^*_3(f) \pi^*_1 (g)),
\end{equation}
where $\pi^*_i$ denotes the usual pullback of functions and $\pi_{i *}$ denotes integration along the fiber. To make this work, one has to impose certain finiteness conditions on $\C$. The simplest, and most restrictive such condition is that $\C$ is \emph{finitary}, which means that $\on{Hom}(M,M')$ and $\on{Ext}^1(M,M')$ are finite sets for any pair of objects $M,M' \in \C$. 

$\H_{\C}$ is spanned by $\delta$-functions $\delta_{[M]}, \; [M] \in \on{Iso}(\mc{C})$ supported on individual isomorphism classes, and the product (\ref{Hallprod}) can be explicitly written
\begin{equation} \label{deltamult}
\delta_{{[M]}} \star \delta_{{[N]}} = \sum_{ {R} \in \on{Iso}(\C) } \mathbf{P}^R_{M,N} \delta_{{R}} 
\end{equation}
where 
\[
\mathbf{P}^R_{M,N}  := \# \vert \{ L \subseteq R, L \simeq N, R/L \simeq M \} \vert
\]
The number
\[
\mathbf{P}^R_{M,N} \vert \on{Aut}(M) \vert \vert \on{Aut}(N) \vert
\]
counts the isomorphism classes of short exact sequences of the form
\begin{equation} \label{ses}
0 \rightarrow N \rightarrow R \rightarrow M \rightarrow 0,
\end{equation}
where $ \on{Aut}(M) $ is the automorphism group of $M$.  Thus, product in $\H_{\C}$ encodes the structure of extensions in $\C$. 

Important examples of finitary categories are $\C=\on{Rep}(Q,\mathbb{F}_{q})$ - the category of representations of a quiver $Q$ over a finite field $\mathbb{F}_q$, and $\C = Coh(X/\mathbb{F}_q)$ - the category of coherent sheaves on a smooth projective variety $X$ over $\mathbb{F}_q$. In these examples, the structure constants of $\on{H}_{\C}$ depend on the parameter $q$, and $\on{H}_{\C}$ recovers (parts of) quantum groups and their generalizations.

The basic recipe above extends more generally to the case where $\C$ is a finitary Quillen exact category (see \cite{H}). Exact categories can be viewed as a strictly full extension-closed subcategories of Abelian categories, and can be equivalently described in terms of classes $(\mathfrak{M}, \mathfrak{E})$ of \emph{admissible mono/epi-morphisms}. For example, the category of vector bundles (i.e. locally free sheaves) on a smooth projective curve $X/\mathbb{F}_q$ is an exact category, being an extension-closed full subcategory of $Coh(X/{\mathbb{F}_q})$, with $(\mathfrak{M}, \mathfrak{E})$ consisting of those monos/epis which are locally split. It is not Abelian, since the kernel/cokernel of a morphism of locally free sheaves may be a coherent sheaf that is not locally free. In this example, the stack $\on{Iso}(\C)$ is the domain of definition of automorphic forms for general linear groups over the function field $\mathbb{F}_q (X)$, and the Hall multiplication encodes the action of Hecke operators. Here, the theory makes contact with the Langlands program over function fields (for more on this, see the beautiful papers \cite{K1,KSV}).

\subsection{Hall algebras in a non-additive setting}

A closer examination of the basic construction of $\H_{\C}$ outlined above shows that the assumption that $\C$ be additive is unnecessary. All that is needed to make sense of the Hecke correspondence (\ref{Hecke_correspondence}) used to define $\H_{\C}$ is a category with a well-behaved notion of exact sequences. In the important paper \cite{DK}, the notion of \emph{proto-exact category} is introduced as a non-additive generalization of the notion of Quillen exact category above, and shown to suffice for the construction of an associative Hall algebra. As in the additive case, such a category is defined in terms of a pair $(\mathfrak{M}, \mathfrak{E})$ of admissible mono/epis which are required to satisfy certain properties. The simplest example of a non-additive proto-exact category is the category $\Set$ of finite pointed sets, with $\mathfrak{M}$ all pointed injections, and $\mathfrak{E}$ those pointed surjections which are isomorphisms away from the base-point. 

Many examples of non-additive proto-exact categories $\C$ arise in combinatorics. Here, $\on{Ob}(\C)$ typically consist of combinatorial structures equipped with operations of "inserting" and "collapsing" of sub-structures, corresponding to $(\mathfrak{M}, \mathfrak{E})$. Examples of such $\C$ include trees, graphs, posets, semigroup representations in $\Set$, quiver representations in $\Set$ etc. (see \cite{KS, Sz1, Sz2, Sz3, Sz4} ). The product in $\H_{\C}$, which counts all extensions between two objects, thus amounts to enumerating all combinatorial structures that can be assembled from two others. Here $\H_{\C}$ is typically (dual to) a combinatorial Hopf algebra in the sense of \cite{LR}. Many combinatorial Hopf algebras arise via this mechanism, including the Hopf algebra of symmetric functions, the Connes-Kreimer Hopf algebras of rooted trees and Feynman graphs, and many others. 

\subsection{ The Waldhausen $\mathcal{S}$-construction and K-theory of proto-exact categories}

It is a natural question what advantages, if any, there are to thinking of combinatorial objects in terms of proto-exact categories and Hall algebras. The answer, as evident in other forms of categorification, is that certain constructions are only visible at the categorical level. In \cite{DK}, the authors associate to a proto-exact category $\C$ a simplicial groupoid $\SS_{\bullet} \C$, called the \emph{Waldhausen $\SS$-construction}, where $\SS_{n} \C$ is closely related to $\fl_n$ above. $\SS_{\bullet} \C$ is shown to have a number of very interesting properties, including the structure of a \emph{2-Segal} space - a form of higher associativity, of which $\H_{\C}$ is but a shadow. This structure was also studied in the papers \cite{GKT1, GKT2, GKT3} from a somewhat different perspective.  

As explained in \cite{DK}, $\SS_{\bullet} \C$ may be used to define the algebraic K-theory of $\C$, by
\begin{equation} \label{kth_def}
K_n (\C) = \pi_{n+1} | \SS_{\bullet} \C |
\end{equation}
where $|\SS_{\bullet} \C |$ denotes the geometric realization of $S_{\bullet} \C$. These groups appear to contain interesting homotopy-theoretic information, even for very simple categories like $\Set$, as evidenced by the following result:

\begin{mythm}[\cite{CLS, Dei}] \label{Kset}
 $K_* (\Set) \simeq \pi^s_{*}(\mathbb{S})$,  where the right hand side denotes the stable homotopy groups of the sphere spectrum.
\end{mythm}

\subsection{Matroids as a proto-exact category}

Matroids are combinatorial structures which abstract different notions of independence encountered across mathematics. A matroid $M$ consists of a finite set $E_M$ (the \emph{ground set}) together with a collection of  \emph{independent subsets} $I \subseteq 2^{E_M}$ satisfying certain natural properties. The prototypical example is obtained by taking $E_M$ to be a set of vectors in some vector space $V$, and taking $I \subseteq 2^{E_M}$ to be those subsets which are linearly independent. 
Matroids and their generalizations have found a vast array of applications across several areas of mathematics, such as for instance in tropical geometry, where \emph{valuated matroids} play the role of linear spaces. 

Matroids form a category $\MS$ with respect to \emph{strong maps}, which are a generalization of linear map in this setting, and it is the category $\MS$ that forms the object of study in this paper. Other aspects of $\MS$ have also been studied in \cite{HP}.  For technical reasons, we prefer to work with \emph{pointed matroids}, where the ground set $E_M$ is pointed by a distinguished element.  We show 

\begin{nothma}[\S \ref{MS_strong} and \S \ref{MS_Bmodules}] \label{mat_proto_thm}
The category $\MS$ of pointed matroids and strong maps has the structure of a finitary proto-exact category, with $\mathfrak{M}$ matroid restrictions and $\mathfrak{E}$ matroid contractions.
\end{nothma}

We give two proofs of this theorem, which reduces to verifying the existence of certain special pushouts/pull-backs in $\MS$. The first is written in the "classical" language of matroids, while the second uses a description of $\MS$ in terms of embedded semi-modules over the Boolean semiring $\mathbb{B}$ given in \cite{CGM}. We are hopeful that a description of valuated matroids along the lines of \cite{CGM} can be given, and that our proof should  generalize to that situation as well. 

We proceed to define and study the algebraic K-theory of $\MS$ via Definition \ref{kth_def}. $\MS$ has an exact forgetful functor to $\Set$ possessing an exact left adjoint sending $E \in \Set$ to the "free pointed matroid on $E$" . These can be used to relate the $K_*(\Set)$ and $K_*(\MS)$. We show:

\begin{nothmb} [Theorem \ref{theorem: K_0}] \label{K0_mat}
$K_0 (\MS) \simeq \mathbb{Z} \oplus \mathbb{Z}$
\end{nothmb}

and

\begin{nothmc} [Theorem \ref{theomre: K_i injection}] \label{K_inj}
There are injective group homomorphisms $$\pi^s_n(\mathbb{S}) \simeq K_n (\Set) \hookrightarrow K_n (\MS)$$ for all $n \geq 0$. 
\end{nothmc}

This shows in particular that $K_n(\MS)$ is in general non-trivial for $n \geq 0$. 
\medskip

As a corollary of Theorem A, we are able to define the  Hall algebra $\H_{\MS}$.  The product $[M] \star [N]$ in this algebra enumerates all matroids $[L]$ such that $L \vert S \simeq M$ and $L/S \simeq N$ for some subset $S \subseteq E_L$. $\H_{\MS}$ turns out to be dual to combinatorial Hopf algebra introduced by Schmitt in \cite{Schmitt}, called the \emph{matroid-minor Hopf algebra}. We obtain the following:

\begin{nothmd} [Theorem \ref{theorem: hopf and hall}]\label{Hall_alg}
$\H_{\MS}$ has the structure of a graded, connected, co-commutative Hopf algebra, dual to Schmitt's matroid-minor Hopf algebra. 
\end{nothmd}

\subsection{Outline of this paper}

In section \ref{proto_hall} we recall the basics of proto-exact categories and their Hall algebras following \cite{DK}. Basic notions regarding pointed matroids are laid out in section \ref{matroids_intro}. Sections \ref{MS_strong} and \ref{MS_Bmodules} contain two proofs of Theorem A - one in classical matroid-theoretic language and one using the language of $\mathbb{B}$-modules introduced in \cite{CGM}. In section \ref{Kth} we define the K-theory of $\MS$ and prove Theorems B and C. Finally, in section \ref{MM_Hopf} we relate the Hall algebra $\H_{\MS}$ to Schmitt's matroid-minor Hopf algebra, proving Theorem D.

\vspace{0.9cm}

\noindent \textbf{Acknowledgments}\\

The third author is grateful to Tobias Dyckerhoff for explanations regarding the paper \cite{DK} and for the support of a Simons Foundation Collaboration Grant. 

\section{Proto-exact categories and their Hall algebras} \label{proto_hall}

In this section, we recall the notion of a proto-exact category $\E$ following \cite{DK}, where we direct the interested reader for details and proofs. This is a generalization of a Quillen exact category that allows $\E$ to be non-additive, and yet provides enough structure to define an associative Hall algebra by counting certain distinguished exact sequences in $\E$. 
As usual, we denote monomorphisms in $\E$ by $\hookrightarrow$ and epimorphisms by $\twoheadrightarrow$.

 A commutative square
\begin{equation} \label{comm_square}
			\xymatrix{
				A \ar[r]^{i} \ar[d]^{j} & B\ar[d]^{j'} \\
				A' \ar[r]^{i'} &  B'.
			}
\end{equation}
is called biCartesian if it is both Cartesian and co-Cartesian. 

\begin{mydef}\label{proto_exact} A {\em proto-exact} category is a category $\E$ equipped with two classes of morphisms $\MM$, $\EE$, called {\em admissible monomorphisms} and {\em admissible epimorphisms} respectively. The triple $(\E, \MM, \EE)$ is required to satisfy the following properties:
	\begin{enumerate}
		\item The category $\E$ has a zero object $0$. Any morphism $0 \rightarrow A$ is in $\MM$, and any morphism $A \rightarrow 0$ is in $\EE$.
		\item The classes $\MM, \EE$ are closed under composition and contain all isomorphisms.
		\item A commutative square \eqref{comm_square} in $\E$ with $i,i' \in \MM$ and $j, j' \in \EE$ is Cartesian iff it is co-Cartesian. 
		\item Every diagram in $\E$ 
		\[ \xymatrix{ A' \ar@{^{(}->}[r]^{i'} & B' & \ar@{->>}[l]_{j'} B} \]
		with $i' \in \MM$ and $j' \in \EE$ can be completed to a biCartesian square (\ref{comm_square}) with $i \in \MM$ and $j \in \EE$. 
		\item  Every diagram in $\E$ 
		\[ \xymatrix{ A'  & \ar@{->>}[l]_{j} A \ar@{^{(}->}[r]^{i}  &B} \]
		with $i \in \MM$ and $j \in \EE$ can be completed to a biCartesian square \eqref{comm_square} with $i' \in \MM$ and $j' \in \EE$. 		
\end{enumerate}
\end{mydef}

A biCartesian square of the form with the horizontal maps are in $\MM$ and the vertical maps are in $\EE$
		\[
			\xymatrix{
				A \ar@{^{(}->}[r] \ar@{->>}[d] & B\ar@{->>}[d] \\
				0 \ar@{^{(}->}[r] &  C
			}
		\]
		is called an {\em admissible short exact sequence}, or, an {\em admissible extension of $C$ by $A$}, and will also be written as
		\begin{equation} \label{ses}
		\ses{A}{B}{C}
		\end{equation}
		We will denote the object $C$ (unique up to a unique isomorphism) by $B/A$. A functor $F: \C \mapsto \mathcal{D}$ between proto-exact categories is \emph{exact} if it preserves admissible short exact sequences. 
		Two extensions $A \hookrightarrow B \twoheadrightarrow C$
and $A \hookrightarrow B' \twoheadrightarrow C$ of $C$ by $A$ are called equivalent if there exists a
commutative diagram
\[
	\xymatrix{ A \ar[d]_{id} \ar@{^{(}->}[r] & B\ar[d]^{\cong} \ar@{->>}[r]& C \ar[d]^{id}\\
	A\ar@{^{(}->}[r] & B' \ar@{->>}[r]& C.}
\]
and the set of equivalence classes of such will be denoted by  $\on{Ext}_{\E}(C,A)$. Two admissible monomorphisms $i_1: A \hookrightarrow B$ and $i_2 : A' \hookrightarrow B$ in $\MM$ are \emph{isomorphic} if there exists an isomorphism $f: A \rightarrow A'$ such that $i_1= i_2 \circ f$. We call the isomorphism classes in $\MM$ \emph{admissible subobjects}.

\begin{mydef}\label{defi:finitary} A proto-exact category $\E$ is called {\em finitary} if, for every pair of objects
	$A$,$B$, the sets $\on{Hom}_{\E}(A,B)$ and $\on{Ext}_{\E}(A,B)$ have finite cardinality.
\end{mydef}

\begin{myeg} \label{example: proto-exact}
	\begin{enumerate}
		\item Any Quillen exact category is proto-exact, with the same exact structure. In particular, any Abelian category $\E$ is proto-exact with $\MM$ all monomorphisms and $\EE$ all epimorphisms respectively. The category $\on{Rep}(Q, \mathbb{F}_q)$ of representations of a quiver $Q$ over a finite field $\mathbb{F}_q$, and $Coh(X / \mathbb{F}_q)$ - the category of coherent sheaves on a smooth projective variety over $\mathbb{F}_q$ are both finitary Abelian.		
		\item The simplest example of a non-additive proto-exact category is the category $\Set$ whose objects are pointed sets with pointed maps as morphisms. Here $\MM$ consists of all pointed injections, and $\EE$ all pointed surjections $p: (S,*) \rightarrow (T,*)$ such that $p \vert_{S \backslash p^{-1}(*)}$ is injective. The full subcategory $\Set^{fin}$ of finite pointed sets is finitary. 
		\end{enumerate}
\end{myeg}

\subsection{The Hall algebra} \label{Hall_alg_sec}

Let $\E$ be a finitary proto-exact category, and $k$ a field of characteristic zero. Define the Hall algebra $\H_{\E}$ over $k$ as
\[
\H_{\E} := \{ f: \on{Iso}(\E) \rightarrow k \vert f \textrm{ has finite support } \},
\]
where $\on{Iso}(\E)$ denotes the set of isomorphism classes in $\E$. $\H_{\E}$ is an associative $k$--algebra under the convolution product
\begin{equation} \label{hall_product}
f \bullet g ([B]) := \sum_{A \subseteq B} f([B/A])g([A]),
\end{equation}
where the summation $\sum_{A \subseteq B}$ is taken over isomorphism classes of admissible sub-objects $i: A \hookrightarrow B, i \in \MM$, and
$[A]$ etc. denotes the isomorphism class of $A$ in $\E$. Note that this sum is finite, since $
\E$ is assumed finitary.  $\H_{\E}$ has a basis consisting of delta-functions $\delta_{[B]}$, $[B] \in \on{Iso}(\E)$, where
\[
\delta_{[B]} ([A]) =\begin{cases} 
      1 & A \simeq B \\
      0 & \textrm{ otherwise }.
   \end{cases}
\]
The multiplicative unit of $\H_{\E}$ is given by $\delta_{[0]}$.
The structure constants of this basis are given by
\[
\delta_{[A]} \bullet \delta_{[C]} =  \sum_{[B] \in \on{Iso}(\E)} g^B_{A,C} \delta_{[B]},
\]
where
\[
g^B_{A,C} = \#  \{ D \subseteq B \vert D \simeq C, B/D \simeq A \} .
\]
In other words, $g^B_{A,C}$ counts the number of admissible subobjects $D$ of $B$ isomorphic to $C$ such that $B/D$ is isomorphic to $A$. 
The \emph{Grothendieck group} of $\E$, denoted $K_0(\E)$ is defined as the free group on $\on{Iso}(\E)$ modulo the relations $[B]=[A][C]$ for every admissible short exact sequence
(\ref{ses}). When $\E$ admits split admissible short exact sequences of the form $$ \ses{A}{A\oplus B}{B}$$ $K_0 (\E)$ is Abelian, and has the familiar description
\[
K_0 (\E) = \mathbb{Z} [\on{Iso}(\E) ] / \sim,
\]
where $\sim$ is generated by the relations $[B] = [A]+[C]$ for all admissible short exact sequences (\ref{ses}). We denote by $K_0(\E)^+ \subseteq K_0 (\E)$ the sub-semigroup generated by the effective classes. $\H_{\E}$ is naturally graded by $K_0 (\E)^+$, with $deg(\delta_{[B]}) = [B] \in K_0 (\E)^+$.

Whether $\H_{\E}$ carries a co-product making it into a bialgebra depends on further properties of $\E$. For instance, if $\E$ is finitary, Abelian, linear over $\mathbb{F}_q$ and hereditary, $\H_{\E}$ carries the so-called \emph{Green's co-product} (see \cite{S}). In this paper, we will be concerned with situations where $\E$ is not additive, and where the following alternative construction applies. Suppose that $\E$ is equipped with finite direct sums, and that the only admissible sub-objects of $A \oplus B$ are of the form $A' \oplus B'$, where $A' \subseteq A, B' \subseteq B$. In this case, we may define
\[
\Delta: \H_{\E} \rightarrow \H_{\E} \otimes \H_{\E}
\]
by
\begin{equation} \label{coprod-def}
\Delta(f)([A], [B]) = f([A \oplus B]).
\end{equation}
$\Delta$ is easily seen to be compatible with the associative product  $\bullet$ on $\H_{\E}$. It is clear from (\ref{coprod-def}) that $\Delta$ is co-commutative, and that the subspace of primitive elements of $\H_{\E}$ is spanned by $\{ \delta_{[B]} \}$, where $B$ is indecomposable (i.e. cannot be written as a non-trivial direct sum of sub-objects). By the Milnor-Moore theorem, any graded connected co-commutative bialgebra is a Hopf algebra, isomorphic to the enveloping algebra of its primitive elements. To summarize, we have:

\begin{mythm}\label{Hall_theorem}
Let $\E$ be a finitary proto-exact category $\E$ satisfying the condition $$C \subseteq A \oplus B \Rightarrow C \simeq A' \oplus B', A' \subseteq A, B' \subseteq B.$$ Then $\H_{\E}$ has the structure of a $K^+_0(\E)$-graded, connected, co-commutative Hopf algebra over $k$ with co-product (\ref{coprod-def}). $\H_{\E} \simeq \mathbf{U}(\delta_{{[B]}})$, where $B$ is indecomposable.
\end{mythm}


\begin{myeg}
The category $\E=\Set$, also known as the category of  finite dimensional vector spaces over ``the field with one element $\mathbb{F}_1$'' has direct sums (defined as wedge sums), and satisfies the conditions of the theorem. In this case $\H_{\E} \simeq k[x]$, with $$ \Delta(x) = x\otimes 1 + 1 \otimes x. $$ 
\end{myeg}

\section{Matroids and Strong Maps} \label{matroids_intro}

This section provides a short introduction to the basic terminology
and results of matroid theory we use in this paper.  For more details,
the reader is encouraged to see \cite{Ox}.

Matroids are combinatorial abstractions of various properties of
linear independence among finitely many vectors in a vector space.  As
such, these objects admit a number of ``cryptomorphic''
axiomatizations; that is, there are a variety of ways of formulating
the axioms for matroids, any of which can be translated to any other.
We present one such formulation, namely the \emph{flats operator}.
Given a finite set $E$ (the \emph{ground set} of our matroid), a
function $\sigma:2^E\to2^E$ is a \emph{flats operator} for a matroid
on $E$ when it satisfies the following axioms:
\begin{enumerate}
\item[(F1)] For all $S\subseteq E$ we have $S\subseteq\sigma S$.
\item[(F2)] The map $\sigma$ is idempotent (i.e.\
  $\sigma=\sigma\sigma$).
\item[(F3)] For all $S\subseteq T\subseteq E$ we have
  $\sigma S\subseteq\sigma T$.
\item[(F4)] For all $S\subseteq E$ and all $x\in E$ and
  $y\in\sigma(S\cup\{x\})\setminus\sigma S$ we have
  $x\in\sigma(S\cup\{y\})$.
\end{enumerate}
Sets of the form $\sigma S$ for some $S\subseteq E$ are called
\emph{flats}.

\begin{rmk}
  Note that (F1), (F2), and (F3) are the axioms of a \emph{closure
    operator} on $E$.  Thus a flats operator on $E$ is a closure
  operator on $E$ satisfying (F4).
\end{rmk}

\begin{myeg}\label{flats examples}
  The following are the prototypical examples of flats
  operators.\footnote{Not all matroids arise in this way.  For more
    details, see \cite{Ox}.}
  \begin{enumerate}
  \item\label{vectorial flats} Let $E$ be a finite subset of a vector
    space $V$.  For all $S\subseteq E$ let
    $\sigma S=\langle S\rangle\cap E$ where $\langle S\rangle$ denotes
    the the subspace of $V$ spanned by $S$.  Then $\sigma$ is a flats
    operator on $E$.
  \item\label{graphic flats} Let $\Gamma$ be a graph with a finite edge
    set $E$.  For all $S\subseteq E$ let $\sigma S$ be the set all
    edges $e$ for which there is a path in $\Gamma$ connecting the
    ends of $e$ and using only edges of $S$.  Then $\sigma$ is a flats
    operator on $E$.
  \end{enumerate}
\end{myeg}

\begin{rmk}
  We note the following:
  \begin{enumerate}
  \item In light of Example \ref{flats examples}, this structure on $E$ abstracts the notion of ``span'' in a vector space.
  \item Given a flats operator $\sigma$ for a matroid $M$ on $E$, one
    can construct $\mathcal{F}(M)$, the set of all flats determined by
    $\sigma$.  The set of flats can also be used to axiomatize
    matroids ``cryptomorphically,'' but the details are unimportant
    for the results of this paper.  We use this fact implicitly below
    when we define the restriction and contraction.
  \end{enumerate}
\end{rmk}

In what follows, for a matroid $M$ on a finite set $E$, we let $\mathcal{F}(M)$ be the set of all flats of $M$. 

\begin{mydef}\label{restriction-contraction}
  Let $M$ be a matroid on $E$ and $S\subseteq E$.
  \begin{enumerate}
  \item The \emph{restriction} of $M$ to $S$ is the matroid $M|S$ on
    $S$ with flats
    \[
      \mathcal{F}(M|S)=\set{A\cap S}{A\in\mathcal{F}(M)}.
    \]
  \item The \emph{contraction} of $M$ by $S$ is the matroid $M/S$ on
    $E\setminus S$ with flats
    \[
      \mathcal{F}(M/S)
      =\set{A\setminus S}{S\subseteq A\in\mathcal{F}(M)}.
    \]
  \end{enumerate}
\end{mydef}

\begin{myeg}\label{restriction examples}
  Let $M$ be a matroid on a ground set $E$ and $S\subseteq E$.
  \begin{enumerate}
  \item\label{vectorial restriction} If $M$ is obtained from
    $E\subseteq V$ for a vector space $V$ as in Example \ref{flats
      examples}, then $M|S$ is obtained likewise
    from $S$ in $\langle S\rangle$.  In particular, matroid
    restriction corresponds to restriction to a subspace in the
    context of vector spaces.
  \item\label{graphic restriction} If $M$ is obtained from $E$ as the
    edge set of a graph $\Gamma$ as in Example \ref{flats
      examples}, then $M|S$ is obtained likewise
    from the graph $\Gamma|S$ with only the edges of $S$.
  \end{enumerate}
\end{myeg}

\begin{myeg}\label{contraction examples}
  Let $M$ be a matroid on a ground set $E$ and $S\subseteq E$.
  \begin{enumerate}
  \item\label{vectorial contraction} If $M$ is obtained from
    $E\subseteq V$ for a vector space $V$ as in Example \ref{flats
      examples}, then $M/S$ is obtained likewise
    from $E\setminus S$ in the vector space $V/\langle S\rangle$.  In
    particular, matroid contraction corresponds to the quotient by a
    subspace in the context of vector spaces.
  \item\label{graphic contraction} If $M$ is obtained from $E$ as the
    edge set of a graph $\Gamma$ as in Example \ref{flats
      examples}, then $M/S$ is obtained likewise
    from the contracted graph $\Gamma/S$ in which the edges of $S$ are
    contracted.
  \end{enumerate}
\end{myeg}

\begin{mydef}
The direct sum $M_1 \oplus M_2 $ of matroids $M_1$ and $M_2$ is defined as the matroid on the ground set $E_{M_1} \sqcup E_{M_2}$ with flats operator $\sigma_1 \sqcup \sigma_2$. In other words, the flats of $M_1 \oplus M_2$ are of the form $F_1 \sqcup F_2$ where $F_1 \in \mathcal{F}(M_1)$ and $F_2 \in \mathcal{F}(M_2)$.  
\end{mydef}

\begin{myeg}\label{direct_sum_exampless}
\begin{enumerate}
\item If $M_1, M_2$ are obtained from subsets $E_i \subseteq V_i, \; i=1,2$ for vector spaces $V_1, V_2$, then $M_1 \oplus M_2$ is obtained from $E_1 \sqcup E_2 \subseteq V_1 \oplus V_2$.
\item If $M_1, M_2$ are graphic, corresponding to graphs $\Gamma_1, \Gamma_2$, then $M_1 \oplus M_2$ is the graphic matroid of the disjoint union $\Gamma_1 \sqcup \Gamma_2$. 
\end{enumerate}
\end{myeg}

A \emph{loop} of a matroid is an element of the ground set which
belongs to the flat $\sigma\emptyset$.\footnote{In a matroid
  determined by a graph as in Example \ref{flats
    examples}.\ref{graphic flats}, this is exactly a loop of the
  graph.}  A \emph{pointed matroid} is a pair $(M,*_M)$ where $M$ is a
matroid and $*_M$ is a distinguished loop (a.k.a.\ the \emph{point}).
We think of the point as a zero element.  We will often suppress the
point in our notation and say that $M$ is a pointed matroid with a point
$*_M$.

\begin{rmk}
When dealing with pointed matroids $(M,*_M)$, we adopt the following conventions to simplify notation, and to ensure that the result of various operations is also a pointed matroid. Recall that the coproduct of pointed sets $(S_1, p_1), (S_2, p_2)$ is $S_1 \vee S_2$ - their wedge sum, in which the basepoints get identified. 
\begin{enumerate}
 \item We will denote by $\wt{E}_M$ the \emph{non-zero elements} in the ground set. Thus for a pointed matroid $(M,*_M)$ we have $E_M=\wt{E}_M \sqcup \{ *_M \}$.
 \item For $S \subseteq \wt{E}_M$,  $M \vert S$ will denote the restriction $M \vert (S \cup \{ *_M \})$. Since $*_M \notin S$, the definition of $M/S$ needs no modification. 
 \item If $(M_1, *_{M_1}), (M_2, *_{M_2})$ are pointed matroids, then $M_1 \oplus M_2$ is defined as the pointed matroid on the ground set $E_{M_1} \vee E_{M_2}$ with the flats operator $\sigma_1 \vee \sigma_2$. 
 \end{enumerate}
\end{rmk}


Having the notion that pointed matroids are generalizations of vector
spaces, one naturally seeks an appropriate translation of linear maps
to this setting.  

\begin{mydef}\label{def:pointed_strong}
Let $M$ and $N$ be pointed matroids on ground sets
$E_M$ and $E_N$ respectively and having flats $\mathcal{F}_M$ and
$\mathcal{F}_N$ respectively.  A \emph{(pointed) strong map} of pointed
matroids $f:M\to N$ is a function $f:E_M\to E_N$ such that
$f(*_M)=*_N$ and for all $A\in\mathcal{F}_N$ we have
$f^{-1}A\in\mathcal{F}_M$.
\end{mydef}

\begin{myeg}\label{example maps}
  The following are the prototypical examples of strong
  maps.\footnote{Not all strong maps between matroids arise in this
    way (even if the matroids themselves arise in this fashion).}
  \begin{enumerate}
  \item\label{vectorial maps} Let $M$ and $N$ be pointed matroids
    arising from $k$-vector spaces $V_M$ and $V_N$ respectively for
    some common field $k$ as in Example \ref{flats
      examples} above.  Every linear map
    $L:V_M\to V_N$ determines a (pointed) strong map $M\to N$.
  \item\label{graphic maps} Let $\Gamma$ and $\Lambda$ be graphs with
    distinguished loops $*_\Gamma$ and $*_\Lambda$, and let
    $(M_\Gamma,*_\Gamma)$ and $(M_\Lambda,*_\Lambda)$ be the pointed
    matroids arising from $\Gamma$ and $\Lambda$ as in Example
    \ref{flats examples} above.  Every graph
    morphism $f:\Gamma\to\Lambda$ preserving the points yields a
    pointed strong map $M_\Gamma\to M_\Lambda$.
  \end{enumerate}
\end{myeg}

The following is immediate:

\begin{pro}
  Pointed matroids and strong maps together form a category $\MS$.
\end{pro}

There is a forgetful functor $$\mathbb{F}: \MS \mapsto \Set$$ which takes a pointed matroid $(M,*_M)$ to its underlying pointed ground set $(E_M, *_M)$.

\section{$\MS$ as a proto-exact category} \label{MS_strong}

Our goal in this section is to show that $\MS$ has the structure of a
proto-exact category in the sense of \cite{DK}. We begin by exhibiting
the classes of admissible monos/epis. 

\begin{mydef}
  Let $\MM$ consist of all strong maps in $\MS$ that can be factored
  as $$ N \overset{\sim}{\rightarrow} M \vert S \hookrightarrow M,$$
  and $\EE$ consists of all strong maps in $\MS$ that can be factored
  as $$ M \twoheadrightarrow M/S \overset{\sim}{\rightarrow} N $$ for
  some $S \subseteq E_M$.
\end{mydef}

Throughout this section, we let $\E:=\MS$. We show that $(\E,\MM,\EE)$
as above is a proto-exact category.
We first prove certain basic lemmas which will be used in what
follows. For the notational convenience, for a function $f:A \to B$
and $S\subseteq A$, we will write $fS$ for the image $f(S)$ whenever
there is no possible confusion. All matroids are assumed to be pointed unless otherwise stated. 


\begin{lem}\label{iso}
  Let $M$ (resp.~$N$) be a pointed matroid on $E_M$ (resp.~$E_N$). A
  function $f:E_M\to E_N$ is an isomorphism in $\E$ precisely when $f$
  is a pointed bijection satisfying
  \[
    A\in\mc{F}(M)\iff fA\in\mc{F}(N).
  \]
\end{lem}

\begin{proof}
  Let $f:E_M\to E_N$ be a set map.
	
  \textit{Sufficiency}: Suppose $f$ is an isomorphism of pointed
  matroids $M$ and $N$.  Applying the forgetful functor $\FF$ we see
  that $f$ is necessarily a pointed bijection; moreover, the
  underlying map of the inverse map $f^{-1}$ is the inverse of the
  underlying map.  Thus both $f$ and $f^{-1}$ are strong maps; given
  $fA\in\ \mc{F}(N)$ we have $f^{-1}(fA)=A\in\mc{F}(M)$, and similarly
  given $A\in\mc{F}(M)$ we have $(f^{-1})^{-1}A=fA\in\mc{F}(N)$.
	
  \textit{Necessity}: Suppose $f$ is a pointed bijection with
  $A\in\mc{F}(M)$ if and only if $fA\in\mc{F}(N)$.  Notice that we
  need only see $f$ and $f^{-1}$ are strong maps to conclude.  Given
  $F\in\mc{F}(N)$ we have $F=f(f^{-1}F)$, so $f^{-1}F\in\mc{F}(M)$ by
  our assumption; thus $f$ is a strong map.  Moreover, given
  $F\in\mc{F}(M)$ we have $(f^{-1})^{-1}F=fF\in\mc{F}(N)$; thus
  $f^{-1}$ is a strong map, and $f$ is an isomorphism.
\end{proof}




Notice that given subsets $T\subseteq S\subseteq E_M$ we have the
following:
\[
  \mc{F}((M|S)/T) =\set{(F\cap S)\setminus T}{T\subseteq
    F\in\mc{F}(M)}.
\]

\begin{lem}\label{iso restriction and contraction}
  Let $f:M\to N$ be an isomorphism in $\E$.
  \begin{enumerate}
  \item For all $\ast_M\in S\subseteq E_M$ the map $f|_S:M|S\to N|fS$
    is an isomorphism.
  \item For all $S\subseteq E_M\setminus\{\ast_M\}$ the map
    $f|_{E_M\setminus S}:M/S\to N/fS$ is an isomorphism.
  \end{enumerate}
\end{lem}

\begin{proof}
  Let $f:M\to N$ be an isomorphism in $\E$.  Note that
  $f(\ast_M)=\ast_N$, so we need not worry about contracting or
  deleting the point as $f$ is a bijection.
	
  \textit{Restriction}: We have the following string of equivalent
  statements via Definition \ref{restriction-contraction}:
  \begin{align*}
    A\in\mc{F}(M|S)
    \iff & \exists B\in\mc{F}(M)\textrm{ such that } A=B\cap S\\
    \iff & \exists B\in\mc{F}(M)\textrm{ such that } fA=fB\cap fS\\
    \iff & \exists B'\in\mc{F}(N)\textrm{ such that }fA=B'\cap fS \\
    \iff & fA\in\mc{F}(N|fS)
  \end{align*}
  Moreover $f|_S$ is a pointed bijection.  Hence $f|_S$ is an
  isomorphism by Lemma \ref{iso}.
	
  \textit{Contraction}: Similarly, we have the following equivalent
  statements via Definition \ref{restriction-contraction}:
  \begin{align*}
    A\in\mc{F}(M/S)
    \iff & \exists B\in\mc{F}(M)
           \textrm{ such that } S\subseteq B\textrm{ and } A=B\setminus S \\
    \iff & \exists B\in\mc{F}(M)
           \textrm{ such that } fS\subseteq fB \textrm{ and } fA=fB\setminus fS\\
    \iff & \exists B'\in\mc {F}(N)
           \textrm{ such that }fS\subseteq B' \textrm{ and }fA=B'\setminus fS \\
    \iff & fA\in\mc{F}(N/fS)
  \end{align*}
  Moreover $f|_{E_M\setminus S}$ is a pointed bijection.  Hence
  $f|_{E_M\setminus S}$ is an isomorphism by Lemma \ref{iso}.
\end{proof}

\begin{lem}\label{minors commute}
  For every matroid $M$ and every $T\subseteq S\subseteq E_M$ we have
  $(M|S)/T=(M/T)|S$.
\end{lem}

\begin{proof}
  It is a standard exercise that $(M\setminus A)/B=(M/B)\setminus A$.
  Apply this result with $A=E\setminus S$ and $B=T$.
\end{proof}

\begin{lem}
  Let $f:M\to N$ be a matroid strong map. Then we have the following:
  \begin{enumerate}
  \item $f$ is a monomorphism in $\E$ precisely when $f$ is injective.
  \item $f$ is an epimorphism in $\E$ precisely when $f$ is
    surjective.
  \end{enumerate}
\end{lem}

\begin{proof}	
  The necessity must hold for both statements simply by noting that
  every strong map is also a function on underlying sets, and thus the
  required properties for monics and epics must hold by the
  corresponding properties of their underlying maps.
	
  Assume $f:M\to N$ is monic and consider the pointed matroid
  $U^\ast_{1,1} = (\{1,\ast\},\ast)$ with flats $\{\ast\}$ and
  $\{1,\ast\}$.  Trivially every pointed map $g:U^\ast_{1,1}\to M$ is
  a strong map.  For each $a\in E_M$ let $g_a:U^\ast_{1,1}\to M$
  denote the pointed map sending $1\mapsto a$.  Suppose $f(a)=f(b)$
  for some $a,b\in E_M$; thus $fg_a=fg_b$ yields $a=g_a(1)=g_b(1)=b$
  by the assumption that $f$ is monic.  Hence $f$ is injective as
  desired.
	
  Assume $f:M\to N$ is epic and consider the pointed matroid
  $U_{0,1}^\ast=(\{1,\ast\},\ast)$ with the flat $\{1,\ast\}$.
  Trivially every pointed map $h:N\to U_{0,1}^\ast$ is a strong map.
  For each $a\in N$ define $h_{a}:N\to U_{0,1}^\ast$ by $x\mapsto 1$
  precisely when $x=a$.  If $f$ is not surjective, then choosing any
  $a\in E_N\setminus fE_M$ we have $h_af=h_\ast f$, and thus
  $h_\ast=h_a$ by the assumption that $f$ is epic.  But this implies
  $a=\ast\in fE_M$, which is absurd.  Hence $f$ is surjective as
  desired.
\end{proof}

\begin{lem}\label{canonical maps}
  Let $M$ be a matroid on $E_M$ and $S\subseteq E_M$.
  \begin{enumerate}
  \item If $\ast_M\in S$, then there is a canonical map
    $i_S:M|S\hookrightarrow M$ in $\E$.
  \item If $\ast_M\notin S$, then there is a canonical map
    $c_S:M\twoheadrightarrow M/S$ in $\E$.
  \end{enumerate}
\end{lem}

\begin{proof}
  \textit{Restriction}: Suppose $\ast_M\in S$ and let $i_S$ denote the
  inclusion $S\hookrightarrow E_M$.  Note that for all $F\in\mc{F}(M)$
  we have $i_S^{-1}F = S\cap F\in\mc{F}(M|S)$ by Definition
  \ref{restriction-contraction}.  Hence $i_S$ is a strong map as
  desired.
	
  \textit{Contraction}: Suppose $\ast_M\notin S$ and let
  $c_S:E_M\to E_M\setminus S$ denote the map defined by
  $c_SS=\{\ast\}$ and $c_S|_{E_M\setminus S}=\id$.  Now for all
  $F\in\mc{F}(M/S)$, there is a flat $A\in\mc{F}(M)$ with
  $S\subseteq A$ and $F=A\setminus S$.  Thus noting $\ast\in F$ yields
  $c_S^{-1}F = S\cup F = S\cup(A\setminus S) = A\in\mc{F}(M)$ as
  $S\subseteq A$.  Hence $c_S$ is a strong map as desired.
\end{proof}

Now, we prove that $(\E,\MM,\EE)$ is a proto-exact category by
verifying each property as below.

\begin{pro}[Verifying Property 1]\label{proposition: property 1}
  $(\E,\MM,\EE)$ is equipped with a zero object.
\end{pro}
\begin{proof}
  The pointed matroid $(\{*\},*)$ is the zero object. Indeed, one
  easily sees that every map $0 \to M$ can be factored as
  $0 \to M|\{\ast\} \hookrightarrow M$ and hence is in
  $\MM$. Similarly, every map $M \to 0$ can be factored as
  $M \twoheadrightarrow M/(E_M\setminus\{\ast_M\}) \to 0$ and hence in
  $\EE$.
\end{proof}


\begin{pro}[Verifying Property 2]\label{proposition: property 2}
  The classes $\MM$ and $\EE$ are closed under composition and contain
  all isomorphisms.
\end{pro}
\begin{proof}
  To see Property 2, first notice that every isomorphism $f:M\to N$
  can be factored as $M \xrightarrow{f} N|E_N \hookrightarrow N$ and
  $M \twoheadrightarrow M/\emptyset \xrightarrow{f} N$; in particular,
  every isomorphism is a member of both $\MM$ and $\EE$.

  Given two composible members $f:M\to N$ and $g:N\to P$ of $\MM$,
  factor these as $M \xrightarrow{f_0} N|S \hookrightarrow N$ and
  $N \xrightarrow{g_0} P|T \hookrightarrow P$.  Now the composite $gf$
  factors as
  $M \xrightarrow{f_0} N|S \hookrightarrow N \xrightarrow{g_0} P|T
  \hookrightarrow P$.  Notice by construction that $g_0S\subseteq T$;
  as $g_0$ is an isomorphism, so Lemma \ref{iso restriction and
    contraction} yields that $g_0|_S: N|S\to P|g_0S$ is an
  isomorphism. In particular, we have the factorization
  $M \xrightarrow{f_0} N|S \xrightarrow{\id} N|S \xrightarrow{g_0|_S}
  P|g_0S \hookrightarrow P$; the first three arrows in this diagram
  are isomorphisms. Hence
  $M \xrightarrow{g_0|_Sf} P|g_0S \hookrightarrow P$ is a
  factorization of $gf$ which shows this is an admissible
  monomorphism.

  Given two composible members $f:M\to N$ and $g:N\to P$ of $\EE$,
  factor these as $M \to M/S \xrightarrow{f_0} N$ and
  $N \to N/T \xrightarrow{g_0} P$.  Now the composite factors as
  $M \to M/S \xrightarrow{f_0} N \to N/T \xrightarrow{g_0} P$.  Notice
  by construction that $f_0$ is an isomorphism, so Lemma \ref{iso
    restriction and contraction} yields that
  $M/(S\cup f_0^{-1}T) \xrightarrow{f_0|_{E_M\setminus(S\cup
      f_0^{-1}T)}} N/T$ is an isomorphism.  In particular we have a
  factorization
  $M \to M/(S\cup f^{-1}T) \xrightarrow{f_0|_{E_M\setminus(S\cup
      f_0^{-1}T)}} N/T \xrightarrow{g_0} P$; the latter two arrows in
  this diagram are isomorphisms.  Hence
  $M \to M/(S\cup f^{-1}T) \xrightarrow{g_0f_0|_{E_M\setminus(S\cup
      f_0^{-1}T)}} P$ is a factorization of $gf$ which shows this is
  an admissible epimorphism.
\end{proof}



To see Properties 3, 4, and 5 we will appeal frequently to the
following Lemma:
\begin{lem}\label{simple bicartesian}
  For all $T\subseteq S\subseteq E_M$ with $\ast_M\in S\setminus T$,
  the following is a biCartesian square in $\E$:
  \[
    \begin{tikzcd}
      M|S \ar{r}{i_S} \ar[swap]{d}{c_T} & M \ar{d}{c_T'}
      \\
      (M|S)/T \ar[swap]{r}{i_S'} & M/T
    \end{tikzcd}
  \]
\end{lem}

\begin{proof}
  Notice trivially that the above square commutes.
	
  \textbf{Cartesian}: Suppose that we have the following commutative
  diagram in $\E$:
  \[
    \begin{tikzcd}
      M|S \ar{r}{i_S} \ar[swap]{d}{c_T} & M \ar[swap]{d}{c_T'}
      \ar{ddr}{\beta} &
      \\
      (M|S)/T \ar{r}{i_S'} \ar[swap]{rrd}{\alpha} & M/T &
      \\
      & & N
    \end{tikzcd}
  \]
  Define $\gamma=\beta|_{E_M\setminus T}:M/T \to N$ and notice
  $\beta=\gamma c_T'$ by construction and the fact that
  $\beta i_s =\alpha c_T$. Furthermore we have
  \[
    \alpha c_T=\beta i_S=\gamma c_T'i_S=\gamma i_S'c_T
  \]
  which yields $\alpha=\gamma i_S'$ after noting trivially that
  $c_T|_{S\setminus T}=\id_{S\setminus T}$.
	
  Let $F\in\mc{F}(N)$ be arbitrary.  By assumption we have that
  $\beta^{-1}F\in\mc{F}(M)$ and $\alpha^{-1}F\in\mc{F}((M|S)/T)$.  Now
  there is an $A\in\mc{F}(M)$ such that $T\subseteq A$ and
  $\alpha^{-1}F=(A\setminus T)\cap S$.  On the other hand
  $\beta i_S=\alpha c_T$ yields the following:
  \[
    S\cap\beta^{-1}F = i_S^{-1}\beta^{-1}F = c_T^{-1}\alpha^{-1}F =
    T\cup\alpha^{-1}F = T\cup((A\setminus T)\cap S)
  \]
  Thus $T\subseteq S\cap\beta^{-1}F$ yields $T\subseteq\beta^{-1}F$,
  and so $\gamma^{-1}F = (\beta^{-1}F)\setminus T\in\mc{F}(M/T)$.
  Hence $\gamma$ is a strong map.  Applying the forgetful functor
  $\FF:\E\to \Set$, note that $\FF\gamma$ is the pushout morphism of
  the corresponding square; $\gamma$ is thus uniquely determined in
  $\E$ by uniqueness of the underlying set map $\FF\gamma$.  Hence the
  square is Cartesian.
	
  \textbf{CoCartesian}: Suppose that we have the following commutative
  diagram in $\E$:
  \[
    \begin{tikzcd}
      N \ar{drr}{\beta} \ar[swap]{ddr}{\alpha} & &
      \\
      & M|S \ar[swap]{r}{i_S} \ar{d}{c_T} & M \ar{d}{c_T'}
      \\
      & (M|S)/T \ar[swap]{r}{i_S'} & M/T
    \end{tikzcd}
  \]
  One observes that $\beta(E_N) \subseteq S$. Indeed, assume to the
  contrary that there is an $x\in E_N\setminus\beta^{-1}S$; thus
  $\beta(x)\in E_M\setminus S$ and we have $\beta(x)=c_T'\beta(x)$ as
  $T\subseteq S$.  Now this yields
  $\beta(x)=i_S'\alpha(x)=\alpha(x)\in S$, contradicting our initial
  assumption.  Hence $\beta(E_N)\subseteq S$.
	
  Define $\gamma:E_N\to S:x\mapsto\beta(x)$ and note that $\gamma$ is
  well-defined by our above argument.  By construction
  $i_S\gamma=\beta$, and
  $i_S'\alpha=c_T'\beta=c_T'i_S\gamma=i_S'c_T\gamma$; hence we have
  $\alpha=c_T\gamma$ by injectivity of $i_S'$.
	
  Let $F\in\mc{F}(M)$ be arbitrary.  Now
  $\gamma^{-1}(F\cap S)=\beta^{-1}(F\cap S)=\beta^{-1}F\in\mc{F}(N)$
  as $\beta$ is a strong morphism.  Hence $\gamma$ is a strong
  morphism.  Applying the forgetful functor $\FF:\E\to \Set$, note
  that $\FF\gamma$ is the pullback morphism of this square; $\gamma$
  is thus uniquely determined in $\E$.  Hence the square is
  coCartesian.
\end{proof}

We can now complete the proof that $\E$ is proto-exact.

\begin{pro}[Verifying Property 4]\label{proposition: property 4}
  Every diagram
  $\xymatrix{ P \ar@{^{(}->}[r]^{i'} & Q & \ar@{->>}[l]_{j'} N}$ in
  $\E$ with $i' \in \MM$ and $j' \in \EE$ can be completed to a
  biCartesian square
  \[
    \xymatrix{
      M \ar[r]^{i} \ar[d]_{j} & N\ar[d]^{j'} \\
      P \ar[r]_{i'} & Q }
  \]
  for some $M \in \E$, $i \in \MM$, $j \in \EE$.
\end{pro}
\begin{proof}
  Since $i' \in \MM$, there exists $S \subseteq E_Q$ such that
  $i'=i_sg$, where $g:P \to Q|S$ is an isomorphism and $i_S:Q|S \to Q$
  (as in Lemma \ref{canonical maps}). Similarly, since $j' \in \EE$,
  there exists $T \subseteq E_N$ such that $j'=fc'_T$, where
  $f:N/T \to Q$ is an isomorphism and $c'_T:N \to N/T$ (as in Lemma
  \ref{canonical maps}). We prove that $M=N|(j')^{-1}(S)$ gives us the
  desired biCartesian square along with a canonical choice of $i$ and
  $j$.

  Let $A:=(j')^{-1}(S)$. Then one can easily see that $T\subseteq A$,
  $*_N \in A\setminus T$ and hence, from the above factorizations of
  $i'$ and $j'$, we obtain the following commuting diagram:
  \[
    \begin{tikzcd}[column sep=large,row sep=large]
      & N|A \ar{r}{i_{A}} \ar{d}{c_T} & N \ar{d}{c_T'}
      \\
      P \ar{r}{f|_{A}^{-1}g} \ar{d}{\id} & (N/T)\vert A \ar{r}{i_{A}'}
      \ar{d}{f|_{A}} & N/T \ar{d}{f}
      \\
      P\ar{r}{g} & Q\vert S \ar{r}{i_S} & Q
    \end{tikzcd}
  \]

  Now note that the following square is a commuting square in $\E$,
  where $i=i_{A}$ and $j=g^{-1}f|_{A}c_T$:
  \[
    \begin{tikzcd}
      N|A \ar{r}{i} \ar[swap]{d}{j} & N \ar{d}{j'}
      \\
      P \ar[swap]{r}{i'} & Q
    \end{tikzcd}
  \]

  Then we have a factorization of $i$ as follows:
  \[
    N|A\xrightarrow{\id}N|A\xrightarrow{i_{A}}N
  \]
  Using Lemma \ref{minors commute}, we have a factorization of $j$:
  \[
    N|A\xrightarrow{c_T}(N|A)/T\xrightarrow{g^{-1}f|_{A}} P.
  \]
  Moreover $g^{-1}f|_{A}$ is an isomorphism by Lemma \ref{iso
    restriction and contraction}.  Hence $i\in\MM$ and $j\in\EE$.

  To see that this square is Cartesian, note that every commuting
  diagram
  \[
    \begin{tikzcd}
      N|A \ar{r}{i} \ar[swap]{d}{j} & N \ar[swap]{d}{j'}
      \ar{ddr}{\alpha} &
      \\
      P \ar{r}{i'} \ar[swap]{rrd}{\beta} & Q &
      \\
      & & M
    \end{tikzcd}
  \]
  determines a corresponding commuting diagram
  \[
    \begin{tikzcd}[row sep=large,column sep=large]
      N|A \ar{r}{i_{A}} \ar[swap]{d}{c_T} & N \ar[swap]{d}{c_T'}
      \ar{ddr}{\alpha} &
      \\
      (N|A)/T \ar{r}{i_{A}'} \ar[swap]{rrd}{\beta g^{-1}f|_{A}} & N/T
      &
      \\
      & & M
    \end{tikzcd}
  \]
  which admits a unique map $\delta:N/T\to M$ such that the diagram
  commutes by Lemma \ref{simple bicartesian}.  On the other hand, this
  implies that $\gamma=\delta f^{-1}$ is the pushout of the original
  square by uniqueness of the pushout in $\Set$.

  To see that this square is coCartesian, note that every commuting
  diagram
  \[
    \begin{tikzcd}
      M \ar{rrd}{\alpha} \ar[swap]{rdd}{\beta} & &
      \\
      & N|A \ar[swap]{r}{i} \ar{d}{j} & N \ar{d}{j'}
      \\
      & P \ar[swap]{r}{i'} & Q
    \end{tikzcd}
  \]
  determines a corresponding commuting diagram
  \[
    \begin{tikzcd}[column sep=large,row sep=large]
      M \ar{rrd}{\alpha} \ar[swap]{rdd}{f|_{A}^{-1}g\beta} & &
      \\
      & N|A \ar[swap]{r}{i_{A}} \ar{d}{c_T} & N \ar{d}{c_T'}
      \\
      & (N|A)/T \ar[swap]{r}{i'_{A}} & N/T
    \end{tikzcd}
  \]
  which admits a unique map $\gamma:M\to N|A$ such that the diagram
  commutes by Lemma \ref{simple bicartesian}.  On the other hand, this
  implies that $\gamma$ is the pullback of the original square by
  uniqueness of the pullback in $\Set$.

  In particular, we have shown that every diagram
  $P\overset{i'}{\hookrightarrow}Q\overset{j'}{\twoheadleftarrow}N$
  with arrows $i'\in\MM$ and $j'\in\EE$ determines a biCartesian
  square with new arrows
  $P\overset{j}{\twoheadleftarrow}M\overset{j}{\hookrightarrow}N$
  where which $i\in\MM$ and $j\in\EE$.

\end{proof}

\begin{pro}[Verifying Property 5]\label{proposition: property 5}
  Every diagram
  $\xymatrix{ P & \ar@{->>}[l]_{j} M \ar@{^{(}->}[r]^{i} & N}$ in $\E$
  with $i \in \MM$ and $j \in \EE$ can be completed to a biCartesian
  square
  \[
    \xymatrix{
      M \ar[r]^{i} \ar[d]_{j} & N\ar[d]^{j'} \\
      P \ar[r]^{i'} & Q.  }
  \]
  for some $Q \in \E$, $i' \in \MM$, $j' \in \EE$.
\end{pro}
\begin{proof}
  Since $i\in\MM$ and $j\in\EE$, we have factorizations of $i$ and $j$
  from which we obtain the following commuting diagram:
  \[
    \begin{tikzcd}[column sep=large,row sep=large]
      M \ar{r}{f} \ar[swap]{d}{c_T} & N|S \ar{r}{i_S} \ar{d}{c_{fT}} &
      N \ar{d}{c_{fT}'}
      \\
      M/T \ar[swap]{d}{g} \ar{r}{f|_{E_M\setminus T}} & (N\vert S)/fT
      \ar{r}{i_S'} \ar{d}{gf|_{E_M\setminus fT}^{-1}} & N/fT
      \\
      P \ar[swap]{r}{\id} & P &
    \end{tikzcd}
  \]
  Now note that the following square is a commuting square in $\E$,
  where $i'=i_S'f|_{E_M\setminus fT}g^{-1}$ and $j'=c_{fT}'$:
  \[
    \begin{tikzcd}
      M \ar{r}{i} \ar[swap]{d}{j} & N \ar{d}{j'}
      \\
      P \ar[swap]{r}{i'} & N/fT
    \end{tikzcd}
  \]
  We have factorizations
  $N\xrightarrow{c_{fT}'}N/fT\xrightarrow{\id}N/fT$ and
  $P\xrightarrow{f|_{E_M\setminus
      fT}g^{-1}}(N|S)/fT\xrightarrow{i_S'}N/fT$.  Moreover
  $f|_{E_M\setminus fT}g^{-1}$ is an isomorphism by Lemma \ref{iso
    restriction and contraction}.  Hence $i'\in\MM$ and $j'\in\EE$.

  To see that this square is Cartesian, note that every commuting
  diagram
  \[
    \begin{tikzcd}
      M \ar{r}{i} \ar[swap]{d}{j} & N \ar[swap]{d}{j'}
      \ar{ddr}{\alpha} &
      \\
      P \ar{r}{i'} \ar[swap]{rrd}{\beta} & N/fT &
      \\
      & & Q
    \end{tikzcd}
  \]
  determines a corresponding commuting diagram
  \[
    \begin{tikzcd}[row sep=large,column sep=large]
      N|S \ar{r}{i_S} \ar[swap]{d}{c_{fT}} & N \ar[swap]{d}{c_{fT}'}
      \ar{ddr}{\alpha} &
      \\
      (N|S)/fT \ar{r}{i_S'} \ar[swap]{rrd}{\beta gf|_{E_M\setminus
          T}^{-1}} & N/fT &
      \\
      & & Q
    \end{tikzcd}
  \]
  which admits a unique map $\gamma:N/T\to Q$ such that the diagram
  commutes by Lemma \ref{simple bicartesian}.  On the other hand, this
  implies that $\gamma$ is the pushout of the original square by
  uniqueness of the pushout in $\Set$.

  To see that this square is coCartesian, note that every commuting
  diagram
  \[
    \begin{tikzcd}
      Q \ar{rrd}{\alpha} \ar[swap]{rdd}{\beta} & &
      \\
      & M \ar[swap]{r}{i} \ar{d}{j} & N \ar{d}{j'}
      \\
      & P \ar[swap]{r}{i'} & N/fT
    \end{tikzcd}
  \]
  determines a corresponding commuting diagram
  \[
    \begin{tikzcd}[column sep=large,row sep=large]
      Q \ar{rrd}{\alpha} \ar[swap]{rdd}{f|_{E_M\setminus
          T}g^{-1}\beta} & &
      \\
      & N|S \ar[swap]{r}{i_S} \ar{d}{c_{fT}} & N \ar{d}{c_{fT}'}
      \\
      & (N|S)/fT \ar[swap]{r}{i'_S} & N/fT
    \end{tikzcd}
  \]
  which admits a unique map $\delta:Q\to N|S$ such that the diagram
  commutes by Lemma \ref{simple bicartesian}.  On the other hand, this
  implies that $\gamma=f^{-1}\delta$ is the pullback of the original
  square by uniqueness of the pullback in $\Set$.
\end{proof}

\begin{pro}[Verifying Property 3]\label{proposition: property 3}
  A commuting square in $\E$ with $i,i'\in\MM$ and $j,j'\in\EE$:
  \[
    \begin{tikzcd}
      M \ar{r}{i} \ar[swap]{d}{j} & N \ar{d}{j'}
      \\
      P \ar[swap]{r}{i'} & Q
    \end{tikzcd}
  \]
  is Cartesian if and only if it is coCartersian.
\end{pro}
\begin{proof}
  Suppose the above square is either Cartesian or coCartesian.  By the
  previous propositions, both
  $P\overset{j}{\twoheadleftarrow}M\overset{i}{\hookrightarrow}N$ and
  $P\overset{i'}{\hookrightarrow}Q\overset{j'}{\twoheadleftarrow}N$
  can be completed to biCartesian squares in $\E$ having all arrows
  from $\MM$ and $\EE$.  On the other hand, pullback and pushout
  objects are unique up to isomorphism.  Thus the original square is
  necessarily biCartesian.  Hence Cartesian and coCartesian are
  equivalent for all such squares.
\end{proof}

Propositions \ref{proposition: property 1}, \ref{proposition: property
  2}, \ref{proposition: property 3}, \ref{proposition: property 4},
and \ref{proposition: property 5} thus complete our first proof of
Theorem A.

\section{$\MS$ as a proto-exact category via $\mathbb{B}$-modules} \label{MS_Bmodules}

In this section, we present another proof showing that $\MS$ is a proto-exact category by appealing to the recent work \cite{CGM} of C.~Crowley, N.~Giansiracusa, and J.~Mundinger. Our motivation of introducing the second proof is to shed some light on generalizing the current work to the case of Hopf algebras for matroids over hyperfields introduced by the authors of the current paper in \cite{EJS}. 

\subsection{Matroids as $\BB$-modules}

In \cite{CGM} the authors give a very useful characterization of matroids and strong maps in terms of $\BB$-modules, where $\BB$ denotes the Boolean semifield $\BB=\{0,1\}$ with 
\[
1 \cdot 1 = 1, \, 0 \cdot 1=1 \cdot 0 = 0 \cdot 0 =0 \; \; \; \; \; 0+0=0, \, 0+1=1+0=1+1=1.
\] 
We proceed to review their construction. For a finite set $E$, we denote by $\BB^E$ the free $\BB$-module on $E$, with standard basis $\{ e_i \}, i \in E$. The linear dual $\on{Hom}_{\BB}(\BB^E, \BB)$ is denoted $(\BB^{E})^{\vee}$, and has dual basis $\{ x_i \}, i \in E$, with $\langle e_i, x_j \rangle = \delta_{ij}$.

\begin{rmk}\label{remark: free module}
In fact, the notion of free $\mathbb{B}$-modules is rather subtle. For instance, any `free $\mathbb{B}$-module' of dimension $n$ does not have to be isomorphic to $\mathbb{B}^n$. See, \cite[\S 2]{MZ} for details. Nonetheless, in this paper, we restrict ourselves to the case of free $\mathbb{B}$-modules of the form $\mathbb{B}^E$. 
\end{rmk}

From now one, all matrids are assumed to be pointed. Given a matroid $M$, let $L_M \subseteq \BB^{E_M}$ be the $\BB$-submodule generated by the support vectors of the cocircuits of $M$. They show:

\begin{pro}[\cite{CGM}]
	Let $M$ be a matroid. Then $M$ is completely determined by the $\BB$-module $L_M$ together with its embedding $L_M \hookrightarrow \BB^{E_M}$. 
\end{pro}

For matroids $N, M$, a map $f: E_N \rightarrow E_M$, induces a $\BB$-module map $$f_*: (\BB^{E_N})^{\vee} \rightarrow (\BB^{E_M}) ^{\vee}$$ $$f_* (x_i) = x_{f(i)}.$$ Taking the transpose (dual) we obtain a $\BB$-module map \begin{equation} \label{ind_map} f^{\vee}_* : \BB^{E_M} \rightarrow \BB^{E_N}. \end{equation}

\begin{pro}[\cite{CGM}]
	$f: E_N \rightarrow E_M$ defines a strong map if and only if $f^{\vee}_*(L_M) \subseteq L_N$.
\end{pro}

\begin{mydef}
	Let $\Emb$ denote the category with:
	\begin{itemize}
		\item Objects are embedded sub-modules $L\subseteq \BB^{E}$ for a finite set $E$. 
		\item  Morphisms from $L \subseteq \BB^E$ to $K \subseteq \BB^F$ are commutative squares 
		\[
		\xymatrix{
			\BB^E \ar@{->}[r] & \BB^F  \\
			L \ar@{->}[r] \ar@{^{(}->}[u] &  K \ar@{^{(}->}[u] 
		}
		\]
	\end{itemize}
\end{mydef}

Then, one has the following:

\begin{pro}[\cite{CGM}] 
	There exists a faithful functor $$ \L: (\MS)^{op} \rightarrow \Emb $$ which assigns to a matroid $M$ the embedded sub-module $L_M \subseteq \BB^{E_M}$ and to a strong map of matroids $f: N \rightarrow M $ the induced map \eqref{ind_map}.
\end{pro}

The functor $\L$ is not full, and as the authors point out in \cite{CGM}, a general morphism in $\Emb$ between two objects in the essential image of $\L$ may be viewed as a "multi-valued" strong map. One also obtains the following pleasant characterization of the restriction and contraction operations:

\begin{pro}[\cite{CGM}]\label{proposition: Noah's minor}
	Let $M$ be a matroid, and $S \subseteq E_M$. 
	\begin{enumerate}
		\item $$\L(M \vert S) = \pi_S(L_M) \subseteq \BB^S, $$ where $$ \pi_S: \BB^{E_M} \rightarrow \BB^S $$  is the canonical projection. Under the functor $\L$, the commutative square
		$$
		\xymatrix{
			\BB^{E_M} \ar@{->}[r]^{\pi_S} & \BB^S  \\
			L_M \ar@{->}[r] \ar@{^{(}->}[u] &  \pi_S(L_M) \ar@{^{(}->}[u] 
		}
		$$
		corresponds to the canonical strong inclusion $ M  \hookleftarrow M \vert S$
		\item $$ \L(M/S) = L_M \cap \BB^{E_M \backslash S} \subseteq  \BB^{E_M \backslash S}, $$ where $\BB^{E_M \backslash S}$ denotes the $\BB$-submodule of $\BB^{E_M}$ consiting of vectors having $0$ in all components corresponding to $S$. Under the functor $\L$, the commutative square
		$$
		\xymatrix{
			\BB^{E_M \backslash S} \ar@{->}[r]^{\iota} & \BB^{E_M}  \\
			L_M \cap \BB^{E_M \backslash S}  \ar@{->}[r] \ar@{^{(}->}[u] &  L_M \ar@{^{(}->}[u] 
		}
		$$
		corresponds to the canonical strong map $M \rightarrow M/S$. 
	\end{enumerate}
\end{pro}



\begin{mythm}
$\MS$ has the structure of a proto-exact category
\end{mythm}

\begin{proof}
$\MS$ is pointed, with $(\{\ast \}, \ast)$ the zero object. We note that $\L((\{ \ast \}, \ast)) =0$ - the trivial $\BB$-module. This verifies property $(1)$ of Definition \ref{proto_exact}. The classes $\MM, \EE$ obviously contain all isomorphisms and are closed under composition, showing $(2)$ of Definition \ref{proto_exact}. To show the existence of the push-outs and pull-backs $(4), (5)$, it suffices to work in the category $\Emb$ via the embedding $\L$, keeping in mind that it is contravariant. 

Consider a diagram in $\MS$ of the kind considered in $(4)$. Applying $\L$, this becomes a diagram in $\Emb$ of the form
\[ \xymatrix{ \pi(L_M \cap \BB^{E_M \backslash T}) \subseteq \BB^{E_M \backslash (T \cup S)} & \ar@{->>}[l]_-\pi L_M \cap  \BB^{E_M \backslash T} \subseteq \BB^{E_M \backslash T} \ar@{^{(}->}[r]^-i  & L_M \subseteq B^{E_M}} \]
where $S, T \subseteq E_M$ are disjoint, and $\pi$ denotes the projection induced by the inclusion of sets $E_M \backslash (S \cup T) \subseteq E_M \backslash T$. The pushout of this diagram in $\Emb$ is easily seen to be 

\[ 
\xymatrix{
 \pi(L_M) \subseteq \BB^{E_M \backslash S} & \ar@{->>}[l]_-\pi L_M \subseteq \BB^{E_M} \\
 \pi(L_M \cap \BB^{E_M \backslash T}) \subseteq \BB^{E_M \backslash (T \cup S)}  \ar@{^{(}->}[u] & \ar@{->>}[l]_-\pi L_M \cap  \BB^{E_M \backslash T} \subseteq \BB^{E_M \backslash T}  \ar@{^{(}->}[u]  }
  \]
For a matroid $M$,  $\pi(L_M) \subseteq \BB^{E_M \backslash S}$ is isomorphic to $\L(M \vert S)$, which shows that the pullback of the diagram $(4)$ exists in $\MS$, and that the completing maps lie in $\MM, \EE$ as desired. 

Similarly, a diagram in $\MS$ of the kind $(5)$, becomes after applying $\L$ a diagram in $\Emb$ of the form
\[
\xymatrix{ \pi(L_M) \cap  \BB^{E_M \backslash (T \cup S)} \subseteq \BB^{E_M \backslash (T \cup S)}  \ar@{^{(}->}[r] & \pi(L_M) \subseteq \BB^{E_M \backslash S} &  \ar@{->>}[l]_-\pi  L_M \subseteq \BB^{E_M}   }
\]
for disjoint $S, T \subseteq E_M$. The pullback of this diagram in $\Emb$ is 
\[ 
\xymatrix{
 \pi(L_M) \subseteq \BB^{E_M \backslash S} & \ar@{->>}[l]_-\pi L_M \subseteq \BB^{E_M} \\
 \pi(L_M \cap \BB^{E_M \backslash T}) \subseteq \BB^{E_M \backslash (T \cup S)}  \ar@{^{(}->}[u] & \ar@{->>}[l]_-\pi L_M \cap  \BB^{E_M \backslash T} \subseteq \BB^{E_M \backslash T}  \ar@{^{(}->}[u]  }
  \]
This shows that the pushout of the diagram $(5)$ exists in $\MS$. Furthermore, comparing $(4), (5)$ shows that property $(3)$ holds. This completes the proof. 
\end{proof}

\begin{rmk}
We note the following:
\begin{enumerate}
\item The admissible sub-objects and quotient objects of $M \in \MS$ correspond respectively to matroids $M \vert S$ and $M/S$ for subsets $S \subseteq \tilde{E}_M$. 
\item The indecomposable objects of $\MS$ are precisely the connected pointed matroids.
\item The admissible sub-quotients of $M \in \MS$ are precisely the pointed minors of $M$. 
\item The forgetful functor $\FF: \MS \mapsto \Set$ is an exact functor of proto-exact categories. 
\end{enumerate}
\end{rmk}

The above proof shows that the biCartesian completions of the diagrams from Definition (\ref{proto_exact}) in $\MS$ are minors of the matroids in the diagrams. Let $\mathcal{M}$ be a collection of pointed matroids which is closed under taking pointed minors, and let $\MS(\mathcal{M})$ denote the full sub-category of $\MS$ generated by objects in $\mathcal{M}$. We then obtain

\begin{mythm} \label{minor_closed_subcat}
$\MS(\mathcal{M})$ has the structure of proto-exact category. It is a full sub-category of $\MS$. 
\end{mythm}

\section{Algebraic K-theory of matroids} \label{Kth}

\subsection{K-theory of proto-exact categories}

We begin by recalling the construction of the algebraic K-theory of a proto-exact category following \cite{DK, Hek}.  Let $\C$ be a
proto-exact category and let $\SS_n = \SS_n(\C)$ denote the maximal groupoid in the category of
diagrams of the form
\begin{equation}\label{eq:sdiag}
		\xymatrix{ 0 \ar@{^{(}->}[r] & A_{0,1} \ar@{^{(}->}[r] \ar@{->>}[d] & A_{0,2}
		\ar@{^{(}->}[r] \ar@{->>}[d] & \dots & A_{0,n-1} \ar@{^{(}->}[r] \ar@{->>}[d]&
		A_{0,n} \ar@{->>}[d]\\
		&0 \ar@{^{(}->}[r] & A_{1,2} \ar@{^{(}->}[r] \ar@{->>}[d] &\dots & A_{1,n-1} \ar@{^{(}->}[r] \ar@{->>}[d]
		& A_{1,n} \ar@{->>}[d]\\
		&& 0 & \ddots &\vdots & \vdots &\\
		&&&& A_{n-2,n-1} \ar@{^{(}->}[r]\ar@{->>}[d] & A_{n-2,n}\ar@{->>}[d]\\
		&& & & 0 \ar@{^{(}->}[r] & A_{n-1,n}\ar@{->>}[d]\\
		&& & & & 0
	}
\end{equation}
where all horizontal maps are in $\mathfrak{M}$ and all vertical maps in $\mathfrak{E}$, and all squares are required to be biCartesian.
For every $0 \le k \le n$, there is a functor
\[
	\partial_k : \SS_n \to \SS_{n-1}
\]
obtained by omitting in the diagram \eqref{eq:sdiag} the objects in the $k$th row and $k$th column and forming the composite of the
remaining morphisms. Similarly, for every $0 \le k \le n$, there is a functor
\[
	\sigma_k : \SS_{n} \to \SS_{n+1}
\]
given by replacing the $k$th row by two rows connected via identity maps and replacing the $k$th
column by two columns connected via identity maps. $\SS_*(\C)$ together with the $\partial_*, \sigma_*$ forms a simplicial object in the category of groupoids. 

\begin{mydef}  \label{kth_def}
The K-theory of $\C$ is defined by
\begin{equation}
K_n (\C) = \pi_{n+1} | \SS_{\bullet} \C |,
\end{equation}
where $|\SS_{\bullet} \C |$ denotes the geometric realization of $\SS_{\bullet} \C$.
\end{mydef}

\begin{rmk}
One may also develop K-theory of proto-exact categories via a version of Quillen's Q-construction. It is shown in \cite{Hek} that this approach leads to isomorphic K-groups. In particular, Theorem \ref{Kset}, whose original proof used the Q-construction, remains valid if the K-theory of $\Set$ is defined as in Definition \ref{kth_def}. 
\end{rmk}

The Grothendieck group $K_0(\C)$ can be described explicitly as the free group on symbols $[A], A \in \on{Iso}(\C)$, modulo the relations $[B] = [A][C]$ for each
admissible short exact sequence
		\[
		\ses{A}{B}{C}.
		\]
When $\C$ admits direct sums for which 
\[
\ses{A}{A\oplus B}{B} \; \textrm{ and } \ses{B}{A \oplus B}{A}
\]
are both admissible, which is the case for $\C=\MS$, $K_0(\C)$ is Abelian, and can be described as the free Abelian group on $[A] \in \on{Iso}(\C)$ modulo the relations $[B] = [A] + [C]$.
K-theory is functorial under exact functors, meaning that an exact functor between proto-exact categories $F: \C \mapsto \mathcal{D}$ induces group homomorphisms
\[
F_*: K_n (\C) \mapsto K_n (\mathcal{D})
\]
compatible with composition. 

\subsection{K-theory of $\MS$}

We begin by calculating the Grothendieck group of $\MS$. Given a non-zero pointed matroid $M$ and $e \in \tilde{E}_M$, we have an admissible short exact sequence
\[
\ses{M\vert e}{M}{M/e} 
\]
Iterating this procedure shows that the class of any $M \in \MS$ can be expressed as a sum of pointed matroids with one-element ground sets. There are two non-isomorphic such matroids, denoted ${\bf a, b}$, where $rk({\bf a})=1$ and $rk({\bf b})=0$ (i.e. ${\bf b}$ is a "non-zero pointed loop"). They span $K_0 (\MS)$ and are easily seen to be independent, since rank and ground set cardinality is additive in admissible short exact sequences. We have thus proved the following:

\begin{mythm}\label{theorem: K_0}
There is an isomorphism $$K_0(\MS) \rightarrow \mathbb{Z} \oplus \mathbb{Z}$$ determined by $$M \rightarrow (rk(M), \vert E_M \vert - rk(M))$$ for each pointed matroid $M$.
\end{mythm}

The forgetful functor $\FF: \MS \rightarrow \Set$ sending a pointed matroid to its ground set has a left adjoint (see \cite{HP})  $\GG: \Set \rightarrow \MS$ sending a pointed set $E$ to the "free pointed matroid on $E$". More precisely, $\GG(E)$ is the pointed matroid whose flats consist of all subsets of $E$ containing the basepoint. We have $\FF \circ \GG = I$, which implies the following:

\begin{mythm}\label{theomre: K_i injection}
There are injective group homomorphisms $$\pi^s_n(\mathbb{S}) \simeq K_n (\Set) \hookrightarrow K_n (\MS)$$ for all $n \geq 0$. 
\end{mythm}

This shows in particular that $K_n(\MS)$ is in general non-trivial for $n > 0$, and contains interesting information of a homotopy-theoretic nature. 

\section{ $\H_{\MS}$ and the Matroid-Minor Hopf algebra}

In this section, we relate the Hall algebra of $\E = \MS$ ( and more generally of the categories $\MS(\mathcal{M})$)  to the Matroid-Minor Hopf algebras introduced by W.~R.~Schmitt in \cite{Schmitt}. We begin by reviewing the latter, adapting to the case of pointed matroids. 

\subsection{The Matroid-Minor Hopf algebra} \label{MM_Hopf}

Let $\mathcal{M}$ be a collection of pointed matroids which is closed under taking
pointed minors and direct sums and $\Miso$ be the set of isomorphism classes
of pointed matroids in $\mathcal{M}$. Let $[M]$ be the isomorphism class of a pointed matroid $M$ in $\mathcal{M}$. $\Miso$ is equipped with a natural commutative monoid structure, via the pointed direct sum, as follows:
\[
  [M_1]\cdot[M_2]:=[M_1\oplus M_2]
\]
and the identity $[(\{ \ast \}, \ast)]$, the equivalence class of the zero pointed
matroid. Let $k[\Miso]$ be the monoid algebra of $\Miso$ over a field
$k$.

In \cite{Schmitt} Schmitt constructs a comultiplication and counit:
\begin{itemize}
\item (Coproduct)
  \[
    \Delta: k[\Miso] \to k[\Miso]\otimes_kk[\Miso],\quad [M] \mapsto
    \sum_{S\subseteq \wt{E}_M} [M|_S] \otimes [M/S].
  \]
\item (Counit)
  \[
    \varepsilon:k[\Miso] \to k, \quad [M]\mapsto
    \left\{ \begin{array}{ll}
              1 & \textrm{if $E_M=\emptyset$}\\
              0& \textrm{if $E_M\neq \emptyset$},
            \end{array}
          \right. 
        \]
\end{itemize}

$k[\Miso]$ carries a natural grading, where $\deg(M,*_M) = \# \wt{E}_M$. 
With the above maps and grading, $k[\Miso]$ becomes a graded connected
bialgebra and hence, from the result of M.~Takeuchi
\cite{takeuchi1971free}, $k[\Miso]$ has a unique Hopf algebra
structure with a unique antipode $S$ given by:
\begin{equation}
  S=\sum_{i \in \mathbb{N}}(-1)^im^{i-1}\circ \pi^{\otimes i}\circ \Delta^{i-1},
\end{equation}
where $m^{-1}$ is a canonical injection from $k$ to $k[\Miso]$,
$\Delta^{-1}:=\varepsilon$, and $\pi:k[\Miso] \to k[\Miso]$ is the
projection map defined by
\[
  \pi|_{k[\Miso]_n} \left\{ \begin{array}{ll}
                          \id & \textrm{if $n \geq 1$}\\
                          0& \textrm{if $n=0$},
                        \end{array} \right.
\]
and extended linearly to $k[\Miso]$. 

\begin{rmk}
The requirement that $\mathcal{M}$ be closed under direct sums is only needed to define the algebra structure. The coalgebra structure requires only that $\mathcal{M}$ be closed under taking minors. If $\mathcal{M}_1 \subseteq \mathcal{M}_2$, then $k[\mathcal{M}_1]$ is a Hopf subalgebra of $k[\mathcal{M}_2]$.
\end{rmk}

The dual Hopf algebra of $k[\Miso]$, denoted $k[\Miso]^*$, is described explicitly in \cite{CS, KRS}. It is shown that $k[\Miso]^*$ is isomorphic to $\{f:  \Miso \rightarrow k \}$, with the product given by the convolution:
\begin{equation} \label{dual_mm_product}
f \diamond g ([M]) = \sum_{S \subseteq \wt{E}_M} f( [M \vert S]) g([M/S])
\end{equation}
and coproduct 
\[
\Delta(f)([M],[N]) := f([M \oplus N])
\]

Comparing this with the Hopf structure of the Hall algebra $\H_{\MS(\mathcal{M})}$ of Section \ref{Hall_alg_sec}, we see that the coproducts agree, and the algebra structures are opposite of each other; i.e. $\bullet = \diamond^{op}$. However, every enveloping algebra possesses an algebra anti-automorphism which fixes the coproduct. We thus obtain:

\begin{mythm} \label{theorem: hopf and hall}
Let $\mathcal{M}$ be a collection of pointed matroids closed under taking pointed minors and direct sums. Then $\H_{\MS{\mathcal{M}}} \simeq k[\Miso]^*$, where $k[\Miso]^*$ denotes Schmitt's matroid-minor Hopf algebra attached to the collection $\mathcal{M}$. $\H_{\MS{\mathcal{M}}} \simeq \mathbb{U}(\delta_{[M]}), \; [M] \in \Miso^{ind}$, where $\Miso^{ind}$ denotes the isomorphism classes of connected pointed matroids in $\Miso$. 
\end{mythm}





\address{\tiny Department of Mathematical Sciences, Binghamton University, Binghamton, NY} \\
\indent \footnotesize{\email{eppolito@math.binghamton.edu}}

\address{\tiny Department of Mathematical Sciences, Binghamton University, Binghamton, NY} \\
\indent \footnotesize{\email{jjun@math.binghamton.edu}}

\address{\tiny Department of Mathematics and Statistics, Boston University, 111 Cumminton Mall, Boaston} \\
\indent \footnotesize{\email{szczesny@math.bu.edu}}

\end{document}